\documentclass[12pt]{article}
\usepackage{amsmath,amssymb,amsthm,amsfonts,graphicx,color,enumerate,float,epsfig,comment}
  \usepackage{tikz}
  \usepackage{tkz-euclide}
\usepackage[title]{appendix}

  \usetikzlibrary{arrows,calc,decorations.pathmorphing,backgrounds,positioning,fit}

\begin{document}

\newcommand{\mb}[1]{\marginpar{\tiny #1 --mb}}
\newcommand{\kf}[1]{\marginpar{\tiny #1 --kf}}
\newtheorem{thm}{Theorem}[section]
\newtheorem{lemma}[thm]{Lemma}
\newtheorem{cor}[thm]{Corollary}
\newtheorem{prop}[thm]{Proposition}
\newtheorem{question}[thm]{Question}
\newtheorem*{mainthm0}{Theorem \ref{log}}

\newtheorem{remark}[thm]{Remark}
\newtheorem{important remark}[thm]{Important Remark}
\newtheorem{definition}[thm]{Definition}
\newtheorem{example}[thm]{Example}
\newtheorem{fact}[thm]{Fact}
\newtheorem{convention}[thm]{Convention}

\def\diam{\operatorname{diam}}
\def\vl{\operatorname{vl}}
\def\svl{\operatorname{svl}}
\def\cal{\mathcal}
\def\R{{\mathbb R}}
\def\Z{{\mathbb Z}}

\def\Mod{{\rm Mod}}
\def\ep{{\varepsilon}}
\def\Isom{{\rm Isom}}
\def\C{{\mathcal C}}
\def\T{{\mathcal T}}
\def\TT{\overline{\mathcal T}}
\def\dTT{\partial \overline{\mathcal T}}
\def\S{{\mathcal S}}
\def\SS{{\overline{\mathcal S}}}
\def\M{{\mathcal M}}
\def\L{{\mathcal L}}
\def\D{\Delta}
\def\s{\sigma}
\def\t{\tau}
\newcommand{\cN}{{\cal N}}
\newcommand{\F}{{\mathbb F}}

\def\<{{\langle}}
\def\>{{\rangle}}
\def\red{\textcolor{red}}

\title{Handlebody subgroups in a mapping class group}

\author{Mladen Bestvina,  Koji Fujiwara\thanks{The
    first author was supported by the National
    Science Foundation grant 1308178. The second author is supported in part by
Grant-in-Aid for Scientific Research (No. 23244005, 15H05739)}} 


\maketitle

\begin{abstract}
Suppose subgroups $A,B < MCG(S)$ in the mapping class group 
of a closed orientable surface $S$ are given
and let $\langle A, B \rangle$ be the subgroup they generate.
We discuss a question by Minsky asking 
when $\langle A, B \rangle \simeq A*_{A \cap B} B$ 
for handlebody subgroups $A,B$.

\end{abstract}

\section{Introduction}
  
  Let $V$ be a handlebody and $S=\partial V$ the boundary surface. 
We have an inclusion of mapping class
groups, $MCG(V) < MCG(S)$. This subgroup is 
called a {\it handlebody subgroup} of $MCG(S)$.
The kernel of the map $MCG(V)
\to Out(\pi_1(V))$ is denoted by $MCG^0(V)$.

If $M=V_+ \cup_S V_-$ is a Heegaard splitting of a closed orientable 
$3$-manifold, we have two handlebody subgroups $\Gamma_{\pm}
=MCG(V_{\pm}) <MCG(S)$ with $S=\partial V_{\pm}$.
Minsky \cite[Question 5.1]{G} asked
\begin{question}\label{question}
When is 
$\langle \Gamma_+, \Gamma_-\rangle <MCG(S)$ equal to the amalgamation 
$\Gamma_+ *_{\Gamma_+ \cap \Gamma_-} \Gamma_-$?
\end{question}

Let $\C(S)$ be the curve graph of $S$ 
and $D_{\pm} \subset {\mathcal C}(S)$  the set of isotopy
classes of simple curves in $S$ which bound disks
in $V_{\pm}$.  The
 {\it Hempel/Heegaard distance} of the splitting
 is defined to be equal to 
$d(D_+, D_-)=\min\{d_{\cal C(S)}(x,y)\mid
x\in D_+,y\in D_-\}$.
The group $MCG(S)$ acts on $\C(S)$ by isometries.
 The stabilizer subgroups of $D_{\pm}$
 in $MCG(S)$ are $\Gamma_{\pm}$.
If  the Hempel distance is sufficiently large, depending only on $S$ ($>3$ suffices if the genus
of $S$ is at least two \cite{J}), then 
$\Gamma_+ \cap \Gamma_-$ is finite \cite{N}.

The following is the main result. It gives a (partial) negative
answer to Question \ref{question}.
\vskip 0.5cm
\noindent
{\bf Theorem \ref{main}.}
{\it
For the closed surface $S$ of genus $4g+1, g \ge 1$ and for any $N>0$
there exists a Heegaard splitting $M=V_+\cup_S V_-$
so that  $\Gamma_+ \cap \Gamma_-$ is trivial, 
$\langle \Gamma_+, \Gamma_-\rangle$ is not equal to 
$\Gamma_+ *  \Gamma_-$,  and 
$d(D_+, D_-) \ge N$.
}

We will prove Theorem \ref{main} by constructing an example.
To explain the idea we first construct a similar example 
in a certain group action on a simplicial tree (Theorem \ref{example.tree}),
then imitate it for the action of $MCG(S)$ on $\C(S)$.

By contrast,  for the subgroups $MCG^0(V_{\pm})=\Gamma_{\pm}^0<\Gamma_{\pm}$,
Ohshika-Sakuma \cite{OS} showed 
\begin{thm}\label{thm.OS}
 If $d(D_+, D_-)$ is sufficiently large (depending on $S$), 
$\Gamma_+^0 \cap \Gamma_-^0$ is trivial and 
$\langle \Gamma^0_+, \Gamma^0_-\rangle =
\Gamma^0_+ * \Gamma_-^0$.
 \end{thm}
 That   $\Gamma^0_+ \cap \Gamma^0_-$ is trivial
 follows from the fact that 
$\Gamma^0_+, \Gamma_-^0$ are torsion free
(attributed to \cite[proof of Prop 1.7]{O} in \cite{OS}).

Here is an alternative proof, suggested by Minsky, that
$\Gamma^0<MCG(S)$ is torsion free for a handlebody $V$.
  Let $f\in \Gamma^0$ be a torsion
element. Since $f$ has finite order, we have a conformal structure on
$S$ invariant by $f$. Moreover, since $f$ extends to $V$, by the
classical deformation theory of Kleinian groups developed by Ahlfors,
Bers, Kra, Marden, Maskit, and Sullivan, we have a unique hyperbolic
structure on $V$ whose conformal structure at infinity is the
prescribed one.  Since the conformal structure is $f$-invariant, so is
the hyperbolic structure.  Moreover, since $f$ acts trivially on
$\pi_1(V)$, each geodesic in $V$ is invariant by $f$. This implies
that $f$ is identity on $V$, hence $f$ is trivial.

\vskip 0.5cm
\noindent
{\bf Acknowledgements.}
We would like to thank Yair Minsky for useful comments. 

\section{Preliminaries}

Let $S$ be a closed orientable surface. The
mapping class group $MCG(S)$ of $S$ is the group of orientation
preserving homeomorphisms modulo isotopy.
 The curve graph $\mathcal C(S)$ has a vertex
for every isotopy class of essential simple closed curves in $S$, and
an edge corresponding to pairs of simple closed curves that intersect
minimally.

It is a fundamental theorem of Masur and Minsky \cite{MM} that the
curve graph is $\delta$-hyperbolic. 
Moreover, they show that an element $F$
acts hyperbolically if and only if $F$ is pseudo-Anosov, and that the
translation length 
$$trans(F)=\lim\frac{d(x_0,F^n(x_0))}n, x_0\in \C(S)$$
of $F$ is uniformly bounded below by a positive
constant that depends only on $S$.
It follows that $F$ has  an invariant quasi-geodesic, called
an {\em axis} denoted by $axis(F)$, 
 whose quasi-geodesic constants depend only on $S$.

A subset $A \subset X$ in a geodesic space  is $Q$-{\it quasi-convex} if any 
geodesic in $X$ joining two points of $A$ is contained in the $Q$-neighborhood
of $A$. 
If $S$ is the boundary of a handlebody $V$, then the set $D\subset \C(S)$
 of curves that bound disks in $V$
   is quasi-convex (i.e. $Q$-quasi-convex for some $Q$), \cite{MM.qc}.

The stabilizer of the set $D$ in $MCG(S)$ is $MCG(V)$,
the mapping class group of the handlebody $V$, i.e.,
the group of isotopy classes of diffeomorphisms of $V$.
 
 Given a $Q$-quasi-convex subset $X$ in $\C(S)$, 
 we define the nearest point projection 
 $\C(S) \to X$.
 The
nearest point projection is not exactly a map, but a {\em coarse map},
since for a given point maybe there is more than one nearest point,
but the set of such points is bounded in diameter, and the bound depends
only on $\delta$ and $Q$, but not on $X$. 

In this paper we often take $axis(F)$ of a pseudo-Anosov element $F$ as $X$.
We may take any $F$-orbit instead of $axis(F)$.
Two pseudo-Anosov elements $F, G$ are independent (i.e.,
$\langle F, G \rangle$ is not virtually cyclic) if and only if
the nearest point projection of $axis(E)$ to $axis(F)$ has a bounded image
(cf. \cite{MM}, \cite{BF1}).

\section{Acylindrical actions}
In this section we discuss the acylindricity of a group action.
This is a key property to prove Theorem \ref{thm.OS}.

 Acylindricity was  introduced  by Sela for group actions on trees and
 extended 
by Bowditch \cite{Bo}.
Suppose $G$ acts on a metric space $X$. The action is {\it acylindrical} if
for given $R>0$ there exist $L(R)$ and $N(R)$  such that for 
any points $v,w\in X$ with $|v-w| \ge L$, 
there are at most $N$ elements $g\in G$ with
$|v-g(v)|, |w-g(w)| \le R$. (Here $|x-y|$ denotes the distance $d(x,y)$.)
Bowditch \cite{Bo} showed that the action of $MCG(S)$ on $\C(S)$
is acylindrical.

The following criterion will be useful.

\begin{lemma}\label{acylindrical}
Suppose $G$ acts on a simplicial tree $X$.
If the cardinality of the edge stabilizers is uniformly bounded
then the action is acylindrical. 
\end{lemma}

\begin{proof}
Assume that every edge stabilizer contains at most $K$ elements. 
Suppose an integer $R>0$ is given. Take $L >>R$.
We will show that if $|v-w| \ge L$ then there are at most 
$(2R+1)K$ elements $g$ with $|v-gv|, |w-gw| \le R$.

Indeed, let $[v,w]$ be the geodesic from $v$ to $w$, and $[v,w]'$ and
$[v,w]''$ be its subsegments after removing the $R$-neighborhood of
$v,w$, and the $2R$-neighborhood of $v,w$, respectively.  Then by the
assumption, $g([v,w]'') \subset [v,w]'$.  Moreover, for an edge $E
\subset [v,w]''$ near the midpoint, $g(E)$ is contained in
$[v,w]''$ and the distance between $E$ and $g(E)$ is at most $R$.  Now
fix such $E$. Then there are elements $h_1, \cdots, h_n \in G'$ with
$n \le 2R+1$, where $h_1=1$, such that for any concerned element $g$,
there exists $h_i$ with $h_i g(E)=E$. But since the stabilizer of $E$
contains at most $K$ elements, there are at most $nK \le (2R+1)K$
distinct choices for $g$.
\end{proof}

To explain the background we quote a main technical result from \cite{OS}
(we will not use this result).

\begin{thm}\label{free}
Let a group $G$ act acylindrically on a $\delta$-hyperbolic space $X$.  Then for
a given $Q>0$, there exists $M>0$ with the following property.  Let
$A, B \subset X$ be $Q$-quasi-convex subsets, and $G_A < stab_G(A),
G_B < stab_G(B)$ torsion-free subgroups.  If $d_X(A,B) \ge M$ then
\\
(1)
$G_A \cap G_B$ is trivial.
\\
(2)
$\langle G_A, G_B\rangle = G_A*G_B$.
\end{thm}

Applying  the theorem to the action of $MCG(S)$ on $\cal C(S)$
with $A=D_+, B=D_-$, and $G_A=\Gamma^0_+ < \Gamma_+=stab(A)$, 
$G_B=\Gamma^0_- < \Gamma_-=stab(B)$, we obtain Theorem \ref{thm.OS}:
 if $d(D_+, D_-)$ is sufficiently large, depending only on $S$, 
then $\Gamma_+^0 \cap \Gamma_-^0$ is trivial and 
 $\langle \Gamma^0_+,
\Gamma^0_-\rangle = \Gamma^0_+ * \Gamma_-^0$.  

To explain the difference between the torsion-free setting of
\cite{OS} and ours, we review the proof of Theorem \ref{free}.  
%
%
We start with an elementary lemma.

\begin{lemma}\label{unique}
Let $X$ be a $\delta$-hyperbolic space and $A,B \subset X$ be 
$Q$-quasi-convex subsets.
Let $\gamma$ be a shortest geodesic between $A$ and $B$.
Then,
\begin{enumerate}[(1)]
\item
for any $x \in \gamma$ with both $d(x,A), d(x,B) > Q+2 \delta$, 
and for any shortest geodesic $\tau$ between $A$ and $B$, 
we have $d(x, \tau) \le 2 \delta$.
\item
Suppose $f$ is an isometry of $X$ with $f(A)=A, f(B)=B$.
Then for any $x \in \gamma$ with both $d(x,A), d(x,B) > Q+2 \delta$, 
we have $d(x,f(x)) \le 4 \delta$. Hence for any $x \in \gamma$, 
we have $d(x,f(x)) \le 2Q+8 \delta$.
\item
Suppose $f$ is an isometry of $X$ with $f(A)=A$.
For $x \in X\backslash A$ let $\sigma$ be a shortest
geodesic from $x$ to $A$.
For an integer $N>0$ assume $d(x,A) \ge Q+ 4\delta N$
and $d(x,f(x)) \le 4 \delta$.
Then for any point $y \in \sigma$ with 
$Q+2\delta <  d(y,A) < d(x,A) - 4\delta N - 2\delta$, we have 
 $d(y,f^i(y)) \le 4 \delta$ for $1\le i \le N$.


\end{enumerate}
\end{lemma}

\begin{proof}
 (1)  Draw a geodesic quadrilateral with $\gamma, \tau$ a pair of
  opposite sides. 
  By $\delta$-hyperbolicity, $x$ must be in the
  $2\delta$-neighborhood of one of the three sides 
  not equal to $\gamma$, which must
  be $\tau$, for otherwise, $d(x,A) \le Q+2\delta$ or $d(x,B) \le
  Q+2\delta$, impossible.

  (2) Put $f(\gamma)=\tau$. Then for a point $x \in
  \gamma$ satisfying the assumption, by (1) there
  is a point $p \in f(\gamma)$ with $d(x,p) \le 2 \delta$.
  But $d(p,f(x)) \le 2\delta$ since $d(x,A)=d(f(x),A)$ and 
  $|d(x,A)-d(p,A)| \le 2\delta$. By triangle inequality $d(x,f(x)) \le 4 \delta$.
   It then implies $d(x,f(x)) \le 4 \delta +2(Q+2\delta)$ for $x
  \in \gamma$ in general.
  
  (3)
  By triangle inequality, for each $1 \le i \le N$, 
  we have $d(x,f^i(x)) \le 4\delta i$.
  Let $q =\sigma \cap A$, and draw a geodesic quadrilateral 
  with the corners $x,q, f^i(q), f^i(x)$.
  Then by $\delta$-hyperbolicity,
  a concerned point $y \in \sigma$ is in the $2\delta$-neighborhood of 
  the side $f^i(\sigma)$,
  hence, as before $d(y,f^i(y)) \le 4 \delta$.
  %
   %
  %
\end{proof}

\begin{proof}[Proof of Theorem \ref{free}]
(1)
Set $L_0=L(4 \delta)+ 2Q + 4 \delta+4\delta N(4\delta)$.  
We fix a constant $M >>2L_0$.
Let $\gamma$ be a shortest geodesic between $A$ and $B$.
Let $|\gamma|$ denote the length of  $\gamma$.
Since $|\gamma| \ge M \ge L_0$, 
we have points
$x_1, x_2 \in \gamma$ such that $d(x_1, x_2) \ge L(4\delta)$ and all four of
$d(x_1,A)$, $d(x_1,B)$, $d(x_2,A)$, $d(x_2,B)$ are $> Q+2\delta$.  If $f \in
G_A \cap G_B$, then both $d(x_1,f(x_1))$, $d(x_2, f(x_2)) \le 4 \delta$ by
Lemma \ref{unique} (2), hence by the acylindricity there are at most $N(4\delta)$
such elements, so the order of $G_A\cap G_B$ is $\le
N(4\delta)$. 
In particular each element in $G_A\cap G_B$ is torsion. 
Since $G$ is torsion free, $G_A\cap G_B$ is trivial.

(2) 
\\
{\it Claim 1}. Let $1 \not= f \in G_A$.
If $d(x, A) \ge L_0$ then $d(x,f(x)) > 4\delta$.
\\
To argue by contradiction assume $d(x,f(x)) \le 4 \delta$.
Let $\sigma$ be a shortest geodesic from $x$ to $A$.
Apply Lemma \ref{unique} (3) with $N=N(4\delta)$. Then for 
each point $y \in \sigma$ with $Q+2\delta < d(y,A) < d(x,A) - 4\delta N(4\delta)
-2\delta$ and each $1 \le i \le N$,  
we have $d(y,f^i(y)) \le 4 \delta$.
Now the subsegment of $\sigma$, except for the end points, 
that those $y$ can belong to 
has length at least $L_0- 4\delta N(4\delta) - 2\delta -(Q+2\delta)
=L(4\delta)+Q$. Taking two points
near each end of the subsegment, they are moved at most $4\delta$
by $1, f, \cdots, f^N$. But by acylindricity there are at most $N(4\delta)$ such elements. Since $N=N(4\delta)$, 
the order of $f$ must be at most $N(4\delta)$, hence trivial, contradiction.

It follows from Claim 1 that if $x \in \gamma$ with $d(x, A) \ge L_0$, then $x$ 
is not contained in $N_{2\delta} (f(\gamma))$.

Similarly, 
\\
{\it Claim 2}. Let $1 \not= f \in G_B$.
If $d(x, B) \ge L_0$ then $d(x,f(x)) > 4\delta$.

\begin{figure}
\centerline{\scalebox{0.6}{\input{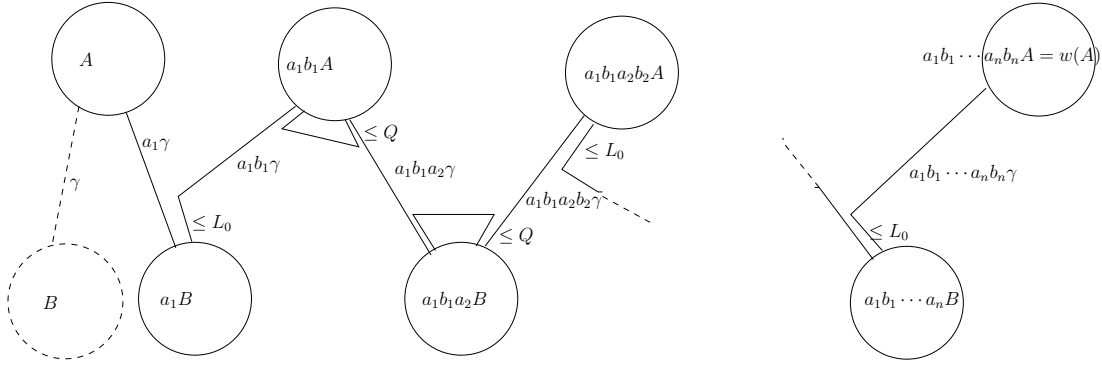}}}
\caption{There is a piecewise geodesic 
 from $A$ to $w(A)$, connecting $a_1\gamma,
a_1b_1\gamma, \cdots, a_1b_1 \cdots a_nb_n\gamma$ in this order, whose length
is at least $|w||\gamma|$. Since the backtrack at each connecting 
point is $\le L_0$, the path is a quasi-geodesic, say $(1.1, 2L_0)$-
quasi-geodesic, since $|\gamma|>>L_0$. In fact its length 
roughly gives a lower bound of the distance between $A$ and $w(A)$.}
\label{fig0}
\end{figure}


Notice that Claim 1 and Claim 2 hold
if every non-trivial element in $G_A, G_B$ 
has order at least $N(4\delta) +1$ or $\infty$.

Now, we apply $1\not=a \in G_A$
to $A \cup \gamma \cup B$, and  obtain $a(B) \cup a(\gamma) \cup A \cup \gamma 
\cup B$. Put $p=\gamma \cap A$. 
The path 
$a(\gamma) \cup [a(p),p]\cup \gamma$  is 
roughly a shortest geodesic from $aB$ to $B$.
This is because $|\gamma| \ge M >>2L_0$, Claim 1, 
and that the geodesic $[a(p),p]$
is contained in the $Q$-neighborhood of $A$.
So, $d(aB, B)$ is at least, say,  $2(|\gamma|-L_0 - 10 \delta)$.
Similarly, now using Claim 2, 
for any $1\not = b \in G_B$, $d(bA, A)$ is at least
$2(|\gamma|-L_0- 10 \delta)$.

To finish, given a reduced word in
$G_A*G_B$, $w=a_1b_1 \cdots a_nb_n$, we let the elements $b_n, a_n, \cdots, b_1, a_1$ successively
act on $A$ (or $B$ if $b_n$ is empty). See Figure \ref{fig0}.
Then as before the distance between $A$
and $w(A)$ is at least, say, $|w|(|\gamma|-2L_0-10\delta)$, where $|w|$ is the length
as a reduced word. (Here we are using a standard fact in $\delta$-hyperbolic geometry that a piecewise
geodesic with each geodesic part long and the ``backtrack'' 
at each connecting point short is not only a quasi-geodesic,
but also a geodesic with the same endpoints follows the path
except for the backtrack parts.)
In particular $A\not= w(A)$, so $w$ is not trivial in $G$.
It implies $\langle G_A, G_B\rangle = G_A*G_B$.   
\end{proof}

There is a more general version of Theorem \ref{free} in which one
does not assume that $G_A$ and $G_B$ are torsion free. To state it we
introduce the following definition. For an isometry $f:X\to X$ define the
{\it coarse fixed set} as $CFix(f)=\{x\in X\mid d(x,f(x))\leq
4\delta\}$.

Also, we will need a version for more than two subsets in $X$ to 
discuss another application. 
For that we introduce one more definition.
Let $A_1, A_2, \cdots, A_n  $ be mutually disjoint
subsets 
in a $\delta$-hyperbolic space $X$.
For a given constant $K>0$, we say that $A_i$ is {\it K-terminal} 
if
for any other $A_j, A_k$ and any shortest geodesic
$\gamma$ between $A_j, A_k$, 
the distance between $A_i$ and $\gamma$ is 
at least $K$. 

\begin{thm}\label{remark.torsion}
Let a group $G$ act acylindrically on a $\delta$-hyperbolic space $X$. Then for
a given $Q>0$, there exists $K>0$ with the following property.  Let
$A_1, \cdots, A_n \subset X$ be $Q$-quasi-convex subsets, and 
$G_{A_i} < stab_G(A_i)$ be subgroups for all $i$. 

Assume that there exists a subset $A_i'$ containing  $A_i$
for each $i$ so that:
\begin{enumerate}[(a)]
\item for every
finite order element $1\neq a\in G_{A_i}$ we have $CFix(a)\subset A_i'$
for each $i$; 
\item $d_X(A_i', A_j')\geq K$ for all pairs $i \not= j$; and 
\item
each $A_i'$ is $K$-terminal. 
\end{enumerate}
Then
\\
(1)
$G_{A_i} \cap G_{A_j}$ is trivial for all $i \not= j$.
\\
(2)
$\langle G_{A_1}, \cdots, G_{A_n}\rangle = G_{A_1}* \cdots * G_{A_n}$.
\end{thm}

The proof is a slight variation of the proof of Theorem \ref{free} and
is omitted. If $G_{A_i}$ does not contain any non-trivial elements
of finite order, then we just put $A_i=A_i'$.
Our counterexamples will have the property that $A_1'\cap
A_2'\neq\emptyset$ (namely, the two coarse fixed sets intersect although
$A_1$ and $A_2$ are far away. 
cf. Claim 1 in the proof of Theorem \ref{free} (2) when $f$ has infinite
order, where $CFix(f)$ is contained in the $L_0$-neighborhood of $A$).

If the sets $A_1', \cdots, A_n'$  satisfy properties (b) and (c) above for 
a constant $K$, 
we say that they are $K$-{\it separated. }

%

\section{Example on a tree}

We will show that Theorem \ref{free} does not hold if we do not 
assume that $G_A$ and $G_B$ are torsion-free. 
We construct a counterexample in the action of $MCG(S)$ 
on $\C(S)$ (Theorem \ref{main}).

To explain the idea  we start with a counterexample
when $X$ is a  simplicial tree.
The key geometric feature is that, if we keep the previous notations,
$\gamma$ and $a(\gamma)$ may stay close along an arbitrarily 
long segment if $a$ has finite order (each point on that segment 
does not move very much by $a$).

\begin{thm}\label{example.tree}
There exists an acylindrical group action on a simplicial tree $X$  by a group
$G$ such that for any number $N>0$
there exist vertices $v, w \in X$ with $|v-w| \ge N$ such that 
$stab_G(v) \cap stab_G(w)$ is trivial
and $\langle stab_G(v), stab_G(w) \rangle$
is not equal to the free product $ stab_G(v) * stab_G(w)$.
\end{thm}

\begin{proof}
We first construct an example with $N=2$.
Start with abelian groups $A, B$ with non-trivial torsion elements $a \in A$
and $b \in B$, for example, $A, B \simeq \Z/2\Z$.

Define the group
$$G=A*_{\langle a\rangle} (\langle a\rangle \times \langle b \rangle)
*_{\langle b \rangle} B$$ and let $T$ be the Bass-Serre tree of this
graph of groups decomposition.

There are two vertices $v,w$ in $T$ at distance two whose stabilizers
are $A$ and $B$. The intersection $A\cap B$ is trivial in $G$ since
$\langle a\rangle \cap \langle b \rangle$ is trivial in $\langle
a\rangle \times \langle b \rangle$. On the other hand, $\<A,B\rangle
=G$ is not equal to $A*B$ since $G$ is the quotient of the free
product $A*B$ by the relation $ab=ba$.  The geometric reason for why
$\<a,b\rangle $ is not equal to $\langle a\rangle*\langle b \rangle$
is that $Fix(a)$ and $Fix(b)$ intersect non-trivially in $T$.

The action on $T$ is acylindrical by Lemma \ref{acylindrical} since
the edge stabilizer is a conjugate of $\langle a\rangle$ or $\langle b
\rangle$.

To produce an acylindrical action that works for all $N>1$
we modify the previous example.
Take the direct product of $\Z=\<t\rangle $ and  the subgroup $\langle a\rangle \times \langle b \rangle$
in $G$.
One can write the new group as 
$$A*_{\langle a\rangle} \{\langle a\rangle\times\langle b \rangle\times \langle t\rangle \} *_{\langle b \rangle} B$$

Further, add a new element $s$ to $A$ with a relation $sa=as$
to get $A'=A\times\<s\>$ and set $C=\langle a\rangle\times \langle b
\rangle \times \langle t\rangle $ and
$$G'=A'*_{\langle a\rangle} C *_{\langle b \rangle} B$$
This is a two edge decomposition.

In the Bass-Serre tree of this decomposition, consider the
``fundamental domain'', i.e. the subtree spanned by two vertices $v,w$
at distance two with stabilizers $A'$ and $B$ respectively. Let
$x$ be the vertex between them with stabilizer $C$. See Figure
\ref{fig1}.

Now consider the ray based at $x$ that contains the vertices $x$,
$t(v)$, $ts(x)$, $tst(v)$, $tsts(x),\cdots$. The stabilizer of every
edge on this ray is $\<a\>$ since both $t$ and $s$ commute with $a$.

So, the intersection of $B$, the vertex group of $w$,
and any of the vertex groups along the ray except for $C$ 
is 
$\langle a\rangle \cap \langle b \rangle =1$.

But for each $n>0$ the subgroup $\<C^{(ts)^n}, B\><G'$
is not equal to $C^{(ts)^n} * B$
since $ a\in C^{(ts)^n}$ and $b\in B$ generate $\<a\>\times\<b\>$ and
not $\<a\>*\<b\>$.

\begin{figure}
\centerline{\scalebox{0.8}{\input{tree.pstex_t}}}
\caption{}
\label{fig1}
\end{figure}

The action of $G'$  is acylindrical by Lemma \ref{acylindrical}
since any edge stabilizer is a conjugate of $\langle a \rangle$
or $\langle b \rangle$.
\end{proof}

\section{Example on $\C(S)$ and proof of theorem}
We will prove the main theorem. 
\begin{thm}\label{main}
For the closed surface  $S$ of genus $4g+1$ with $g \ge 1$  and for any $N>0$
there exists a Heegaard splitting $M=V_+\cup_S V_-$
so that  $\Gamma_+ \cap \Gamma_-$ is trivial, 
$\langle \Gamma_+, \Gamma_-\rangle$ is not equal to 
$\Gamma_+ *  \Gamma_-$,  and 
$d(D_+, D_-) \ge N$.

\end{thm}
We will need two  properties of pseudo-Anosov elements
to prove the theorem (Lemma \ref{elementary}, Lemma \ref{hempel}).

\subsection{Pseudo-Anosov elements by Masur-Smillie}
Let $S$ be a closed surface and 
$F$ a pseudo-Anosov mapping class on $S$.
The {\it elementary closure} of $F$
is the subgroup $E(F)$ of $MCG(S)$ that
consists of mapping classes preserving 
the stable and unstable foliations of $F$. Equivalently, $E(F)$ is the
centralizer of $F$ in $MCG(S)$. 
The group $E(F)$ contains a
unique finite normal subgroup $N(F)$ such that $E(F)/N(F)$ is infinite
cyclic. Note that $E(F^k)=E(F)$ and $N(F^k)=N(F)$ for every $k\neq
0$. 

If $S'\to S$ is a regular cover with deck group $\Delta$ and if
$F:S\to S$ is a pseudo-Anosov mapping class with $N(F)=1$, then we
certainly have $N(F')\supseteq\Delta$ for any lift $F':S'\to S'$ of
any power of $F$, but strict inclusion may hold. It is an interesting
question whether one can construct $F$ so that equality holds for all
regular covers. We call such $F$ {\it prime} and we discuss a
construction of prime pseudo-Anosov mapping classes in Section \ref{prime.section}.
For our purposes we need quite a bit less.

\begin{lemma}\label{elementary}
Suppose $F:S \to S$ is a pseudo-Anosov mapping class
whose stable and unstable foliations have two singular points, 
one of order $p$, the other of
order $q$, with both $p,q$ odd and relatively prime. Let $S'\to S$ be
a double cover with deck group $\<a\>$ and $F'$ a lift of a power
of $F$ to $S'$.
 Then $N(F)=1$ and
$N(F')=\<a\>$.
\end{lemma}

\begin{proof}
We first argue that $N(F)=1$. Suppose $g\in N(F)$. Then $g$ can be
represented by a homeomorphism, also denoted $g:S\to S$, that
preserves both measured foliations. In particular, $g$ is an isometry
in the associated flat metric on $S$ with cone type singularities. The
homeomorphism $g$ fixes both singular points and satisfies both
$g^p=1$ and $g^q=1$, since an isometry that fixes a nonempty open set
is necessarily the identity. Since $p,q$ are relatively prime it
follows that $g=1$.

We now argue that $N(F')=\<a\>$.
We have $\<a\> < N(F')$.
Let $g \in N(F')$, then $g:S'\to S'$ is a finite order
homeomorphism that preserves the lift of stable and unstable foliations
of $F$. 
Composing
with $a$ if necessary we may assume that $g$ fixes both $p$-prong
singularities. 
Arguing as above, we see that 
 $g^p=1$. Since $g^2$ fixes both $q$-prong
 singularities, similarly we have $g^{2q}=1$ and since
$(p,2q)=1$ we have $g=1$.
We showed $N(F')=\<a\>$.
\end{proof}

\begin{cor}\label{MS.elementary}
Pseudo-Anosov mapping classes as in Lemma \ref{elementary} exist in
every genus $\geq 3$.
\end{cor}

\begin{proof}
Write $4g=p+q$ where $p,q$ are relatively prime odd numbers. For
example, we can take $p=2g-1$ and $q=2g+1$. By
the work of Masur-Smillie \cite{MS} $F$ as above exists.
\end{proof}



\subsection{Masur domain and Hempel elements}
Suppose $V$ is a handlebody and 
$S$ its boundary.
Let $D \subset \C(S)$ be the set of
curves that bound disks in $V$.
Denote by $L\subset \mathcal{PML}(S)$ the closure of $D$, viewed as a
subset of $\mathcal{PML}(S)$.
Then $L$ is nowhere dense in
$\mathcal{PML}(S)$ \cite{M}, and its complement $\Omega$ is called
the {\it Masur domain}. 

Hempel \cite{H} found that if the stable lamination of a pseudo-Anosov
element $F$ is in $\Omega$ then $\lim_{n\to \infty}d_{\C(S)}(D,
F^n(D))=\infty$.  
We say a pseudo-Anosov element $F:S\to S$ is {\it Hempel} for $D$
if the nearest point projection of $D$ to $axis(F)$ is a bounded
set.

$F$ is Hempel if and only if the end points of $axis(F)$ are in
$\Omega$, \cite{Sch}. On the other hand, both endpoints of $axis(F)$
are in $L$ if and only if $axis(F)$ is contained in a $K$-neighborhood
of $D$ for some $K>0$ since both $D$ and $axis(F)$ are quasi-convex
subsets in the $\delta$-hyperbolic space $\C(S)$ (cf. \cite{Sch}).
Since $L$ is nowhere dense in $\mathcal{PML}(S)$ and the set of pairs
of endpoints $(\lambda^+,\lambda^-)$ of pseudo-Anosov mapping classes
is dense in $\mathcal{PML}(S)\times \mathcal{PML}(S)$, there is a
pseudo-Anosov element $F$ whose stable and unstable laminations are
not in $L$, so that $F$ is Hempel.

Masur found a condition in terms of the intersection number 
for a curve to be in $D$ \cite[Lemma 1.1]{M} and used it 
to prove $L$ is nowhere dense \cite[Theorem 1.2]{M}.
The following lemma is proved using his ideas.

\begin{lemma}\label{hempel}
Let $V' \to V$ be a double cover between handlebodies with the deck
group $\<a\>$, 
$S'=\partial V', S=\partial V$, and $D' \subset \C(S'), D \subset \C(S)$
the set of curves that bound disks in $V', V$, respectively.

If the genus of $S$ is $\geq 3$, then $MCG(S)$ contains a 
pseudo-Anosov element $F$ such that:
\begin{enumerate}[(i)]
\item $F$ is Hempel for $D$,
\item $F$ lifts to $F':S'\to S'$ and 
$F'$ is Hempel for $D'$,
\item
$N(F')=\langle a \rangle$.
\end{enumerate}

\end{lemma}

\begin{proof}
Let $\Omega$ be the Masur domain for $V$ and $\Omega'$ for $V'$.
We first find a lamination $\Lambda$ on $S$ that is in $\Omega$
such that its lift $\Lambda'$ on $S'$ is also in $\Omega'$.
Choose a pants decomposition of $S$ using
curves in $D$ and a lamination $\Lambda\in\mathcal{PML}(S)$ whose
support intersects each pair of pants in this decomposition in 3 (non-empty)
families of arcs
connecting distinct boundary components (so there are no arcs
connecting a boundary component to itself). In the proof of
\cite[Theorem 1.2]{M} Masur shows that  $\Lambda\in\Omega$
(for example, take the curve $\beta$ in his proof as $\Lambda$). 
This is done by verifying the conditions in Lemma 1.1 for $\beta$
with respect to the pants decomposition in the last two paragraphs
of the proof of Theorem 1.2. Now 
the lift
$\Lambda'$ of $\Lambda$ to $S'$ satisfies the same condition with
respect to the lifted pants decomposition (it lifts since our covering 
is between handlebodies and the boundary curves bound disks), so we have
$\Lambda'\in\Omega'$.

Choose a pseudo-Anosov homeomorphism $G:S\to S$ both of whose fixed
points in $\mathcal{PML}(S)$ are close to $\Lambda$ and in particular
they are in $\Omega$ since $\Omega$ is open.  The lift $G'$ of $G$ (or
its power) to $S'$ similarly has endpoints close to $\Lambda'$ and in
particular in $\Omega'$. It follows that both $G$ and $G'$ are Hempel.

To finish the proof we need to arrange that $G$ has the extra 
property (iii).
Let $H:S\to S$ be
an arbitrary pseudo-Anosov mapping class that satisfies the assumption 
of Lemma \ref{elementary}.  Such $H$ exists 
by Corollary \ref{MS.elementary}. Then
$F=G^nHG^{-n}$ also satisfies the assumptions, and hence also conclusion 
of Lemma \ref{elementary}  for any $n>0$ and has an axis whose
endpoints are close to $\Lambda$ if $n>0$ is sufficiently large.
Therefore $F$ is Hempel, and similarly, the lift $F'$ has an axis
whose endpoints close to $\Lambda'$, therefore $F'$ is Hempel.
\end{proof}

\subsection{Proof of Theorem \ref{main}}
We prove Theorem \ref{main}
by constructing an example. 

\begin{proof}[Proof of Theorem \ref{main}]
Let $H \simeq \Z/2\Z + \Z/2\Z$ with generators $a_1,a_2$,
 and let $V' \to V$ be a normal
cover between handlebodies with the deck group $H$. If $g\geq 2$
is
the genus of $V$, then the genus of $V'$ is $4g-3$.
  Let $D' \subset
\C(S')$ be the set of curves in $S'=\partial V'$ that bound disks in $V'$.
We have two double covers $S' \to S'/a_i$.  Let $D_i \subset
\C(S'/a_i)$ be the set of curves in $S'/a_i$ that bound disks in $V'/a_i$.
Put $S_i=S'/a_i$. The genus of $S_i$ is $2g-1$.
Let $Q$ be a common quasi-convex constant for $D', D_1, D_2$,
and $\delta$ the hyperbolicity constant of $\C(S')$.

Using Lemma  \ref{hempel}, take  a  pseudo-Anosov element $F_i$ on $S_i$
that is Hempel for $D_i$ such that the lift $F_i'$ of $F_i$
to $S'$ is also Hempel for $D'$, and that 
$N(F_i')=\<a_i\>$.
Note that $F_1', F_2'$ are independent 
pseudo-Anosov elements on $S'$ since their elementary closures
are different. 
In particular, the projection of $axis(F_1')$ to $axis(F_2')$ 
is bounded, and vice versa. 

%

 Note that $a_i \in stab(D')$ since $a_i \in H$.
   Set $D_i'=F_i'^N(D')$
for  $N>0$.
Then, $a_i \in stab(D_i')$ since $F_i'$ centralizes $a_i$.

Form the Heegaard splitting $V_+' \cup_{S'} V_-'$ 
such that $D_+=D_1', D_-=D_2' \subset \C(S')$.
The surface $S'$ is fixed but the splitting depends on $N$.
We will argue this is a desired splitting.

Set $\Gamma_i=stab(D_i')<MCG(S')$. In other words,
$\Gamma_1=\Gamma_+, \Gamma_2=\Gamma_-$ in the Heegaard splitting
convention. 
Since $a_i \in \Gamma_i$ and $a_1a_2=a_2a_1$,
$\langle \Gamma_1,\Gamma_2 \rangle$ is not the free product of $\Gamma_1,\Gamma_2$.  

To prove the theorem we are
left to verify $d(D_1',D_2') \to \infty$ as $N \to \infty$ and $\Gamma_1
\cap \Gamma_2=1$ for any large $N>0$.

\begin{figure}
\centerline{\scalebox{0.6}{\input{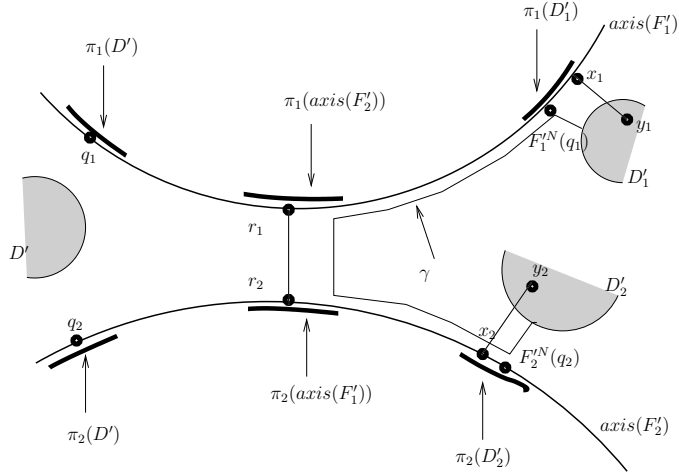}}}
\caption{If $N>0$ is large, the projection $\pi_1(axis(F_2'))$ and $\pi_1(D_1')$
are far apart on $axis(F_1')$, Also, 
$\pi_2(axis(F_1'))$ and $\pi_2(D_2')$
are far apart on $axis(F_2')$. As a consequence 
any shortest geodesic $\gamma$ between $D_1'$ and $D_2'$
must enter a bounded neighborhood of each of those four projection sets,
and the segment in $\gamma$ near $axis(F_i')$ is almost fixed by $a_i$ pointwise.}
\label{fig2}
\end{figure}

 \begin{lemma}\label{apart}
  $d(D_1',D_2') \to \infty$
as $N \to \infty$.
\end{lemma}

\begin{proof}
We claim that there is a  constant $A$ such that for any $N>0$,
$$d(D_1',D_2') \ge (trans(F_1') + trans(F_2'))N-A.$$
Let $\pi_1$ denote the projection to $axis(F_1')$, 
and $\pi_2$ the projection to $axis(F_2')$. As we said they are coarse maps
but we pretend they are maps for simplicity. 
Also, we pretend that both $axis(F_1'), axis(F_2')$ are geodesics.

 Let  $L$ be a common bound of the diameter of the sets 
 $\pi_1(D')$, $\pi_1(axis(F_2'))$, $\pi_2(D')$ and $\pi_2(axis(F_1'))$.
Then $L$ is a bound of $\pi_1(D_1')$ and $\pi_2(D_2')$
for all $N>0$.
Choose points $q_1 \in \pi_1(D')$, $q_2 \in \pi_2(D')$, 
$r_1 \in \pi_1(axis(F_2'))$ and $r_2 \in \pi_2(axis(F_1'))$.

Now assume $N>0$ is so large  that $\pi_1(axis(F_2'))$ and $\pi_1(D_1')$
are far apart, and also $\pi_2(axis(F_1'))$ and $\pi_2(D_2')$ are far apart
(compared to $L$ and $\delta$).
It suffices to show the above inequality  under this assumption. 

Let $y_1 \in D_1'$ and $y_2 \in D_2'$ be any points, and 
put $x_1=\pi_1(y_1), x_2=\pi_2(y_2)$.
Then, by a standard argument using $\delta$-hyperbolicity, the piecewise
geodesic $[y_1, x_1]\cup[x_1,r_1]\cup[r_1,r_2]\cup[r_2,x_2]\cup[x_2,y_2]$
is a quasi-geodesic with uniform  quasi-geodesic constants
that depends only on $L$ and $\delta$. See Figure \ref{fig2}.
Hence the Hausdorff distance between the piecewise geodesic
and the geodesic $[y_1,y_2]$ is bounded (the bound depends only on $L$ and $\delta$).

It follows that there is a constant $C$ such that 
$d(D_1', D_2') \ge d(y_1, x_1)+d(x_1,r_1)+d(r_1,r_2)+d(r_2,x_2)+d(x_2,y_2)-C
\ge d(x_1,r_1)+d(r_2,x_2)-C \ge d(F_1'^N(q_1), r_1)-L + d(F_2'^N(q_2), r_2)-L -C$.
On the other hand since $q_1,r_1 \in axis(F_1')$ and 
$q_2, r_2 \in axis(F_2')$ there is a constant $B$ such that for all $N>0$
 we have $d(F_1'^N(q_1), r_1)+ d(F_2'^N(q_2), r_2)
\ge  (trans(F_1') + trans(F_2'))N-B$.
Combining them we get a desired estimate with $A=2L+B+C$.
\end{proof}

We note that any geodesic joining a point in $D_1'$ and a point
in $D_2'$ passes through a bounded neighborhood of each of $F_1'^N(q_1), 
r_1, r_2, F_2'^N(q_2)$ provided that $N>0$ is large enough. 
The bound depends only on $L$ and $\delta$. See Figure \ref{fig2}.
In the argument we did not use that $D'$ (as well as $D_1', D_2'$) are
quasi-convex.

To argue $\Gamma_1 \cap \Gamma_2=1$, we will need the following lemma from
\cite[Proposition 6]{BF1}. This is a consequence of the fact that $F$
is a ``WPD element''.

\begin{lemma}\label{wpd}
Let $F$ be a pseudo-Anosov mapping class on a hyperbolic surface
$S$. There is a constant $M>0$ such that for any $g\in MCG(S)$ the
diameter of the projection of $g(axis(F))$ to $axis(F)$ in $\C(S)$ is
larger than $L$, then $g\in E(F)$.
\end{lemma}

\begin{lemma}\label{trivial}
$\Gamma_1 \cap \Gamma_2=1$ for any large $N>0$.
\end{lemma}

\begin{proof}
By Lemma \ref{apart}  choose $N$  large such  that $d(D_1',D_2')$
is very large compared to $\delta$ and $L$.
Let $\gamma$ be a shortest geodesic from  $D_1'$ to $D_2'$.
Then as we noted after the proof of Lemma \ref{apart},
$\gamma$ passes through the bounded neighborhood of each of $F_1'^N(q_1),
r_1, r_2, F_2'^N(q_2)$.

Now let $f \in \Gamma_1 \cap \Gamma_2$.  Then we have $d(x,f(x)) \le 2Q+8
\delta$ for any $x \in \gamma$ by Lemma \ref{unique}.  Since all of
$F_1'^N(q_1), r_1, r_2, F_2'^N(q_2)$ are in bounded distance from
$\gamma$ we conclude each of those four points is moved by $f$ a
bounded amount (the bound depends only on $\delta, L, Q$).

But since $F_1'^N(q_1),r_1$ are contained in $axis(F_1')$ and are far apart
from each other 
for any large $N>0$, we find $f(axis(F_1'))$ has a long ($>M$)
projection to $axis(F_1')$, hence $f \in E(F_1')$ 
by Lemma \ref{wpd}. By the same reason
$f \in E(F_2')$.
We conclude $f \in E(F_1') \cap E(F_2')$.
But since $F_1'$ and $F_2'$ are independent,  $f$ must be a torsion element, 
so $f \in N(F_1') \cap N(F_2')$.
By Lemma \ref{elementary},
$f \in \langle a_i \rangle \cap \langle a_2 \rangle =1$.
We showed the lemma. 
\end{proof}

 We proved the theorem. 
\end{proof}


\section{Prime pseudo-Anosov elements}\label{prime.section}
 In view of  Lemma \ref{elementary} we introduce a property
that looks interesting for its own sake.
We say a pseudo-Anosov mapping class $F$ is {\it prime} if its stable/unstable foliations are not
lifts of any foliations of a (possibly orbifold) quotient of $S$.

If $F$ is prime then $E(F)$ is cyclic and $N(F)=1$. Indeed, if
$N(F)\neq 1$ then the two foliations lift from $S/N(F)$, with $N(F)$
realized as a group of isometries of $S$ using Nielsen realization.
Moreover, 

\begin{lemma}\label{elementary2} (cf. Lemma \ref{elementary})
Suppose $S' \to S$ is a finite cover with the 
Deck group $\Delta$.
Let $F$ be a prime pseudo-Anosov element on $S$ and $F'$
a lift of a power of $F$ to $S'$.
Then $N(F')=\Delta$.
\end{lemma}

\begin{proof}
It is clear that $\Delta <N(F')$.
If the inclusion is strict, then 
the stable and unstable foliations of $F$ can be obtained by pulling
back from $S'/N(F')=S/(N(F')/\Delta)$.
So we have a contradiction. 
\end{proof}

Note that if we have a prime pseudo-Anosov 
element on $S$, we can use  Lemma \ref{elementary2}
instead of Lemma 5.2 in the proof of Lemma \ref{hempel} and Theorem \ref{main}.
We will give a construction of prime pseudo-Anosov elements
when the genus of $S$ is 3, so this will also prove the theorem 
for the genus 5 case.

Recall that if $a,b,c,d$ are 4 vectors in $\R^2$ then the cross ratio
is $$[a,b;c,d]=\frac{[a,c][b,d]}{[a,d][b,c]}$$ where
$[x,y]=x_1y_2-x_2y_1$ for $x=(x_1,x_2), y=(y_1,y_2)$.  We do not
define it when one of $[a,c], [b,d], [a,d], [b,c]$ is $0$.  The cross
ratio is invariant under changing signs and scaling individual vectors
and applying matrices in $SL_2(\R)$.  It follows that for any flat
structure on the torus the cross ratio for the vectors in the
directions of four distinct closed geodesics is (well-defined and)
rational.

A singular Euclidean structure (or just a {\it flat} structure) on a
surface $S$ is {\it good} if the cone angle is a multiple of $\pi$ at
each singularity. A geodesic segment connecting two singular points,
or a closed geodesic is {\it good} if the angle along the geodesic at
each singular point is a multiple of $\pi$.

The developing map $\widetilde{S-\Sigma}\to\R^2$ defined on the
universal cover of the complement of the cone points will take a good
geodesic to a straight line, or a line segment. 
So, for any four good 
geodesics, the cross ratio for the four directions, if they 
are distinct,  is well-defined.

Next, if $S'\to S$ is a branched cover between good flat
structures, then the cross ratio of four good geodesics in $S'$ is
equal to the cross ratio of their images in $S$, simply because $S,S'$
have the ``same'' developing map. In particular, all cross ratios
between good geodesics on a torus or a sphere with 4 cone points are
rational, and to prove that a particular good flat
surface is not commensurable with a torus it suffices to produce four
good geodesics whose cross ratio is irrational.

\begin{lemma}\label{prime}
Suppose $F:S\to S$ is a pseudo-Anosov homeomorphism such that:
\begin{enumerate}[(1)]
\item the stable foliation of $F$ has two singular points $x$ and $y$, 
with $p$
  and $q$ prongs respectively, and with $p$ and $q$ distinct odd primes, and

  \item
  a flat structure on $S$ determined by the stable and
  unstable foliations
  has four 
  good closed geodesics with the cross ratio of
  their (distinct) direction vectors 
  $a,b,c,d \in {\mathbb R}^2$ irrational. 
  
\end{enumerate}
Then $F$ is prime.
\end{lemma}

We note that there is a 2-parameter family of flat structures
determined by the two foliations; they depend on the choice of the
transverse measure in a projective class on each foliation. However,
since scaling and linear transformations do not change the cross
ratio, the assumption is independent of these choices.

\begin{proof} 
Let $p,q,F$ be as in the statement.
Now suppose $\pi:S\to S'$ is a branched cover of
degree $d>1$ and $\mathcal F=\pi^{-1}\mathcal F'$. The local degree of
$\pi$ at $x$ is either 1 or $p$. It cannot be 1,
since at any other preimage of $\pi(x)$ the singularity would have to
have $kp$ prongs, and there aren't any. Thus at $x$ the map is modeled
on $z\mapsto z^p$, and similarly at $y$ it looks
like $z\mapsto z^q$. There are now two cases.

\noindent
{\it Case 1.} $\pi(x)\neq \pi(y)$.

It follows that the other points that map to
$\pi(x)$ have 2 prongs and so the map there has local degree 2. Thus $d$
is odd and away from the images of singular points the foliation
$\mathcal F'$ is regular (since otherwise $d$ would have to be
even). Thus there are $(d-p)/2$ other preimages of $\pi(x)$, and
deleting these and the same for the $q$-prong singularity we get that the
Euler characteristic of $S-\pi^{-1}(\{\pi(x),\pi(y)\})$
is $$(2-2g)-(d-p)/2-(d-q)/2-2=-d$$ So the Euler characteristic of the
quotient $S'$ minus 2 singular points is $-1$, i.e. the quotient
is the twice punctured $\R P^2$, which does not support any
pseudo-Anosov homeomorphisms (e.g. the curve complex is finite, see
\cite{scharlemann}). On the other hand, let $\tilde S$ be a finite
cover of the punctured $S$ so that the induced cover to $\R
P^2$ minus two points is regular. Some power of $F$ lifts to $\tilde
F$ on $\tilde S$, and since the cover is regular, a further power
of $\tilde F$
descends to $\R P^2$ (since each element, $a$, of the Deck group
leaves the stable and unstable foliations invariant, so that $(a \tilde F a^{-1})^N=\tilde F ^N$
for some $N>0$, so $a$ and $\tilde F^N$ commute), contradiction.

\noindent
{\it Case 2.} $\pi(x)=\pi(y)$.

Again the other points that map to $\pi(x)=\pi(y)$ are regular and the
map has local degree 2, so there are $\frac {d-(p+q)}2$ such
points. Here $d$ is even and we may have some number, say $k\geq 0$,
of 1-prong singularities $z_1,\cdots,z_k$ in the quotient, with each
singularity having $\frac d2$ preimages where local degree is 2. Now
we have that the Euler characteristic of
$S-\pi^{-1}(\{\pi(x),z_1,\cdots,z_k\})$ is
$$(2-2g)-\frac{d-(p+q)}2-2-\frac{kd}2=-\frac {(k+1)d}2$$ So $k$ is odd
and the Euler characteristic of the quotient $S'$  minus the singular points
is $-\frac{k+1}2$. So, the Euler characteristic of $S'$
is $-\frac{k+1}2 + (k+1)=\frac{k+1}2$.
The only possibilities are $k=1$ and $k=3$, and the
quotient is twice punctured $\R P^2$ or 4 times punctured $S^2$. The
first possibility is ruled out as in Case 1.

 In the second case the good flat structure on $S$ descends to a good flat 
 structure on $S^2$ with 4 singular points, 
 then lifts to a flat structure 
 on the branch double cover $T^2$, with four closed 
 geodesics such that the cross ratio 
 of the four direction vectors is $[a,b;c,d]$ that is not rational, 
 contradiction. 

Indeed, using the same notation as in Case 1, 
a power of $F$ lifts to $\tilde F$ on $\tilde S$ that regularly covers $S^2$
minus 4 points. We lift the flat structure on $S$
and the stable and unstable foliations of $F$
to $\tilde S$. Then their regular leaves are straight lines.
Each deck transformation preserves the foliations, 
so that it is an isometry of $\tilde S$, and that 
the good flat structure on $\tilde S$, with cone 
angle at each singular point at least $2\pi$, descends
to a good flat structure of $S^2$ minus 4 points (and the angle at each puncture is $\pi$). We obtain a good flat structure
on $S^2$ with four good closed geodesics and 
the cross ratio is $[a,b;c,d]$.
Also, the cross ratio will not change when we take a double 
cover that is a flat torus, contradiction. 
\end{proof}

\begin{remark}\label{MS.foliation}
Regarding  the assumption (1), 
if $g$ is the genus of $S$, 
by an Euler characteristic count we must have $p+q=4g$. Conversely,
the Goldbach conjecture predicts that every even integer $>2$ can be
written as a sum of two primes. When the integer is $\geq 8$ and
divisible by 4, the two primes are necessarily distinct and odd. 
For
example, $12=5+7$ satisfies the Goldbach conjecture.
The
work of Masur-Smillie \cite{MS} shows that if $g\geq 3$ and $4g=p+q$
then the surface $S$ admits a pseudo-Anosov homeomorphism
whose stable and unstable foliations have two singular points, 
one of order $p$, the other of
order $q$.
\end{remark}

\begin{example}\label{flat.example}
We now construct an explicit example in genus 3 satisfying the 
assumption of Lemma \ref{prime}.
We take $p=5$, $q=7$. Consider the flat square tiled
surface $S$ pictured below. Edges labeled by the same letter are to be
identified. If the edges are on opposite sides of the parallelogram
they are identified by a translation, and otherwise by a rotation by
$\pi$. The square tiling of $\R^2$ induces one on the surface
$S$. There are two cone points, with cone angles $5\pi$ and $7\pi$
respectively.
So, $S$ has a good flat structure. 

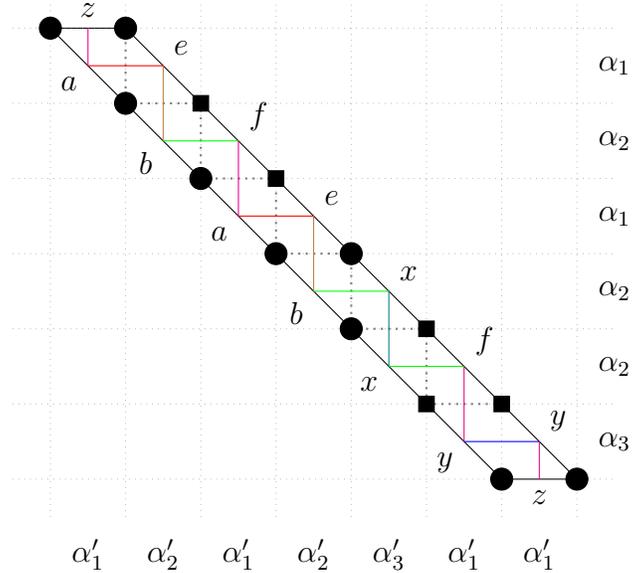
\begin{figure} [ht]
\begin{center}
\begin{tikzpicture}

\draw[black] (0,6) -- (6,0);
\draw[black] (1,6) -- (7,0);
\draw[black] (0,6) -- (1,6);
\draw[black] (6,0) -- (7,0);

\foreach \x in {1,...,5}
{
\draw[gray,thick,dotted] ($ (\x,6-\x) $) -- ($ (\x,7-\x) $);
\draw[gray,thick,dotted] ($ (6-\x,\x) $) -- ($ (7-\x,\x) $);
}

\foreach \x in {0,...,7}
{
\draw[gray,very thin,dotted] ($ (\x,-0.5) $) -- ($ (\x,6.5) $);

}
\foreach \x in {0,...,6}
{
\draw[gray,very thin,dotted] ($ (-0.5,\x) $) -- ($ (7.5,\x) $);
}

\node[scale=0.75] at (0,6) [circle,draw=black,fill=black]{};
\node[scale=0.75] at (1,6) [circle,draw=black,fill=black]{};
\node[scale=0.75] at (1,5) [circle,draw=black,fill=black]{};
\node[scale=0.75] at (2,4) [circle,draw=black,fill=black]{};
\node[scale=0.75] at (3,3) [circle,draw=black,fill=black]{};
\node[scale=0.75] at (4,2) [circle,draw=black,fill=black]{};
\node[scale=0.75] at (6,0) [circle,draw=black,fill=black]{};
\node[scale=0.75] at (7,0) [circle,draw=black,fill=black]{};
\node[scale=0.75] at (4,3) [circle,draw=black,fill=black]{};

\node[scale=0.75] at (5,1) [rectangle,draw=black,fill=black]{};
\node[scale=0.75] at (6,1) [rectangle,draw=black,fill=black]{};
\node[scale=0.75] at (5,2) [rectangle,draw=black,fill=black]{};
\node[scale=0.75] at (3,4) [rectangle,draw=black,fill=black]{};
\node[scale=0.75] at (2,5) [rectangle,draw=black,fill=black]{};

\draw[red] (0.5,5.5) -- (1.5,5.5);
\draw[red] (2.5,3.5) -- (3.5,3.5);

\draw[green] (1.5,4.5) -- (2.5,4.5);
\draw[green] (3.5,2.5) -- (4.5,2.5);
\draw[green] (4.5,1.5) -- (5.5,1.5);

\draw[blue] (5.5,0.5) -- (6.5,0.5);

\draw[magenta] (0.5,5.5) -- (0.5,6);
\draw[magenta] (2.5,3.5) -- (2.5,4.5);
\draw[magenta] (5.5,0.5) -- (5.5,1.5);
\draw[magenta] (6.5,0) -- (6.5,0.5);

\draw[brown] (1.5,4.5) -- (1.5,5.5);
\draw[brown] (3.5,2.5) -- (3.5,3.5);

\draw[teal] (4.5,1.5) -- (4.5,2.5);

\node[below left] at (0.5,5.5) {$a$};
\node[below left] at (1.5,4.5) {$b$};
\node[below left] at (2.5,3.5) {$a$};
\node[below left] at (3.5,2.5) {$b$};
\node[below left] at (4.5,1.5) {$x$};
\node[below left] at (5.5,0.5) {$y$};

\node[above right] at (1.5,5.5) {$e$};
\node[above right] at (2.5,4.5) {$f$};
\node[above right] at (3.5,3.5) {$e$};
\node[above right] at (4.5,2.5) {$x$};
\node[above right] at (5.5,1.5) {$f$};
\node[above right] at (6.5,0.5) {$y$};

\node[above] at (0.5,6) {$z$};
\node[below] at (6.5,0) {$z$};

\node at (7.5,5.5) {$\alpha_1$};
\node at (7.5,4.5) {$\alpha_2$};
\node at (7.5,3.5) {$\alpha_1$};
\node at (7.5,2.5) {$\alpha_2$};
\node at (7.5,1.5) {$\alpha_2$};
\node at (7.5,0.5) {$\alpha_3$};

\node at (0.5,-1) {$\alpha_1'$};
\node at (1.5,-1) {$\alpha_2'$};
\node at (2.5,-1) {$\alpha_1'$};
\node at (3.5,-1) {$\alpha_2'$};
\node at (4.5,-1) {$\alpha_3'$};
\node at (5.5,-1) {$\alpha_1'$};
\node at (6.5,-1) {$\alpha_1'$};

\end{tikzpicture}

 \caption{The square tiled surface. The round vertex has 7 prongs and
   the square vertex has 5. The surface has a good flat structure.}  \label{Fig:Graph}
\end{center}
 \end{figure}

We will use Thurston's construction of pseudo-Anosov homeomorphisms
\cite[Theorem 14.1]{FM} to construct $F$. The lines bisecting the squares
form three horizontal and three vertical geodesics.
The matrix $N$ of intersection numbers, where the $jk$ entry is
the intersection number $i(\alpha_j,\alpha_k')$, is
\[
N=\left(\begin{matrix}
1&1&0\\
1&1&1\\
1&0&0\end{matrix}\right)
\]

The largest eigenvalue of $NN^t$ is $\mu=5.0489\dots$ satisfying the
minimal polynomial $\mu^3-6\mu^2+5\mu-1=0$. Then one can
choose the lengths and heights of the squares (making them into
rectangles, which gives a tiling of $S$) so that the twist in the horizontal multicurve is given by
the matrix 
\[
\left(\begin{matrix}
1&\mu^{1/2}\\
0&1\end{matrix}\right)
\]
and the twist in the vertical multicurve by the matrix
\[
\left(\begin{matrix}
1&0\\
-\mu^{1/2}&1\end{matrix}\right)
\]

This means
that the product of the first and the inverse of the second is
\[
\left(\begin{matrix}
1&\mu^{1/2}\\
0&1\end{matrix}\right)
\left(\begin{matrix}
1&0\\
\mu^{1/2}&1\end{matrix}\right)=
\left(\begin{matrix}
1+\mu&\mu^{1/2}\\
\mu^{1/2}&1\end{matrix}\right)
=A
\]
whose trace is $2+\mu$. So this product is pseudo-Anosov and its
dilatation is the larger, $\lambda$, of the eigenvalues of $A$ and 
the eigenvector is $^t(1, \sigma)$ with $\mu^{1/2}=(1-\sigma^2)/\sigma$
and $\lambda=1/\sigma^2$.

The heights and widths of the rectangles are coordinates of the 
$\mu$-eigenvectors $V$ of $NN^t$ and $V'$ of $N^tN$. We compute
\[
V=\left(
\begin{matrix}
1\\
\mu^2-5\mu+1\\
-2\mu^2+11\mu-4
\end{matrix}
\right)
=
\left(
\begin{matrix}
v_1\\
v_2\\
v_3
\end{matrix}
\right)
\]
and 
\[
V'=\mu^{-1/2}N^tV=
\mu^{-1/2}
\left(
\begin{matrix}
-\mu^2+6\mu-2\\
\mu^2-5\mu+2\\
\mu^2-5\mu+1
\end{matrix}
\right)
=
\left(
\begin{matrix}
v_1'\\
v_2'\\
v_3'
\end{matrix}
\right)
\]

To show that $F$ is prime 
it suffices to find 4 closed geodesics whose slopes have
irrational cross ratio.
We take $a=(1,0)$, $b=(0,1)$, $c=(-v_2',v_1+v_2)$,
$d=(-v_1'-v_2',2v_1+v_2)$, where $c$ and $d$ connect second,
respectively third, vertex on the lower left side in the figure with
the upper right vertex.
They are good closed geodesics on $S$ based at the round vertex
(to compute the angle at the round vertex, it helps first to 
identify the two edges labeled by $z$).
  The cross ratio is 
   \begin{align*}
  \frac{[a,c][b,d]}{[a,d][b,c]} &=\frac{(v_1+v_2)(v_1'+v_2')}{(2v_1+v_2)v_2'}
  =\frac{(\mu^2-5\mu+2)\mu}{(\mu^2-5\mu+3)(\mu^2-5\mu+2)}
  =\frac{\mu}{\mu^2-5\mu+3}\\
  &=\frac{1}{3\mu^2-17\mu+10},
\end{align*}
which is irrational (for the last equality use $\mu^3-6\mu^2+5\mu-1=0$).

\end{example}


\section{Invariable generation}
In this section we discuss another application of 
Theorem \ref{free}.
For this we need the version stated in
Theorem \ref{remark.torsion}.

\subsection{Definitions and results}
Following Dixon \cite{Di} a group $G$ is {\it invariably generated} by a subset
$S$ of $G$ if $G=\langle s^{g(s)}| s\in S\rangle$ for any choice
of $g(s) \in G, s \in S$.
The group $G$ is {\it IG} if it is invariably generated by some subset $S$ in $G$,
or equivalently, if $G$ is invariably generated by $G$. $G$ is {\it FIG}
if is is invariably generated by some finite subset of $G$.
Kantor-Lubotzky-Shalev \cite{KLS} prove
that a linear group is FIG if and only if it is 
finitely generated and virtually solvable. 

Gelander proves that every non-elementary hyperbolic group
is not IG \cite{Ge}.
We generalize this result to acylindrically hyperbolic groups.
A group $G$ is {\it acylindrically hyperbolic} if it admits an acylindrical
action on a hyperbolic space and $G$ is not virtually cyclic \cite{Osin}.
Examples are non-elementary hyperbolic groups, $MCG(S_{g,p}) $
except for the genus $g=0$ and the number of punctures $p \le 3$, 
and $Out(F_n)$ with $n\ge 2$ (cf. \cite{Osin}).

We prove
\begin{thm}\label{notIG}
If $G$ is an acylindrically hyperbolic group, then 
$G$ is not IG.
\end{thm}
In other words, $G$ contains a proper subgroup such that 
any element in $G$ is conjugate to some element in  the subgroup.

\subsection{Elementary facts from $\delta$-hyperbolic spaces}
Suppose $G$ acts on a $\delta$-hyperbolic space $X$.

For $g \in G$, define its {\it minimal translation length} by 
$$min(g)=\inf_{p \in X}|p-g(p)|$$
It is a well-known fact that if $min(g) \ge 10 \delta$ then 
$g$ is hyperbolic. (To be precise, we assume $\delta >0$ here.)

For $L>0$ define a subset 
$$X(g,L)=\{x \in X| |gx-x| \le L\}$$
and put 
$$M(g)=X(g, min(g)+1000\delta)$$
$M(g)$ is a $g$-invariant non-empty set.
If $g$ is hyperbolic, then $M(g)$ is contained in a Hausdorff
neighborhood of an axis of $g, axis(g)$.
This is an easy exercise and we leave it to the reader. 

\begin{lemma}\label{center}
If $g$ has a bounded orbit, then $min(g) \le 6 \delta$.
\end{lemma}

\begin{proof}
This is well known too. 
Let $Z$ be the set of centers of the orbit of a point $x$ by $g$.
$Z$  is
invariant by $g$ and its diameter is at most 
$6\delta$, therefore is  contained in $X(g,6\delta)$.
\end{proof}

\begin{lemma}\label{sublevel.set}
If $L\ge min(g)+1000\delta$, then 
$X(g,L)$ is $100\delta$-quasi-convex.
\end{lemma}
\begin{proof}
We define a function on $X$ as follows:
$t_g(x)=|x-g(x)|$.
\\
{\it Case 1}. $min(g) \le 10 \delta$.
\\
Fix $p \in X(g,10\delta)\subset X(g,L)$.
Suppose $x \in X(g,L)$ is given.
We will show that  $[p,x]$ is contained in the $20 \delta$-neighborhood
of $X(g,L)$.
If $|x-p|\le 100 \delta$, then 
by triangle inequality, $t_g(z) \le 210 \delta$ for every $z \in [p,x]$, 
so that $[p,x] \subset X(g,L)$.
So assume $|x-p|>100\delta$.
To compute $t_g(z)$ for $z \in [p,x]$, 
draw a triangle $\Delta$ for $p,x,g(x)$, 
and let $c \in [p,x]$ be a branch point of this triangle,
i.e., the distance to each side of $\Delta$ from $c$
is at most $\delta$.
Since $|p-g(p)|\le 10 \delta$, $t_g(z) \le 20\delta$
for any point $z\in[p,c]$. For  $z\in [c,x]$, $t_g(z)$ is roughly equal
to $2d(c,z)$, with an additive error
at most $20 \delta$. Also it  is roughly maximal at $x$ on $[p,x]$.
(Imagine the case that $X$ is a tree and $p=g(p)$.)
It follows that $[p,x]$ is contained in the  $20 \delta$-neighborhood
of $X(g,L)$. 

Now suppose another point $y\in X(g,L)$ is given.
Then $[x,y]$ is contained in the $\delta$-neighborhood
of $[p,x]\cup [p,y]$, so that $[x,y]$ is contained
 in the $21 \delta$-neighborhood of $X(g,L)$.
\\
{\it Case 2}. $min(g) \ge 10 \delta$.
\\
Then $g$ is hyperbolic. 
To simplify the argument, let's assume that there is a geodesic axis for $g$.
Then for any $x \in X$, 
$$|t_g(x) - \{trans(g)+2 d(x, axis(g))\}| \le 20 \delta$$
To see this let $x' \in axis(g)$ be a nearest point from $x$.
Then the Hausdorff distance between $[x,g(x)]$ and
$[x,x']\cup [x',g(x')]\cup [g(x'),g(x)]$ is at most $5 \delta$,
and the estimate follows. 

It follows from the above estimate that if $ x \in X(g, L)$, then 
$[x,x']$ is contained in the $20\delta$-neighborhood of $X(g,L)$.
For $y \in X(g,L)$, let $y' \in axis(g)$ be a nearest point from another point  $y$
to $axis(g)$. 
Then $[x,x']\cup [x',y']\cup [y',y]$ is contained in the 
 $20\delta$-neighborhood
of $X(g,L)$.
But since $[x,y]$ is contained in the $5\delta$-neighborhood of 
$[x,x']\cup [x',y']\cup [y',y]$, $[x,y]$ is contained in the $25\delta$-neighborhood
of $X(g,L)$.

The argument for the case that $g$ has only a quasi-geodesic (with uniform
quasi-geodesic constants depending only on $\delta$) as an axis
is similar and  we only need to modify the constants in the
argument. We omit the details. 
\end{proof}

Lemma \ref{sublevel.set} implies 
\begin{lemma}\label{min.set}
$M(g)$ is $100\delta$-quasi-convex.
\end{lemma}

Let $\partial X$ denote the boundary at infinity 
of $X$. For a quasi-convex subset $Y \subset X$, 
let $\partial Y\subset \partial X$ be the boundary at infinity of $Y$.

\begin{lemma}\label{boundary.min}
If $p\in \partial M(g) \subset \partial X$, then 
$g(p)=p$.
\end{lemma}
\begin{proof}
Let $\gamma$ be a geodesic ray from a point in $M(g)$ that 
tends to $p$. Then the ray is contained in the $10\delta$-neighborhood
of $M(g)$. So every point of the ray is moved by $g$ by a bounded amount,
therefore $p$ is fixed by $g$.
\end{proof}

When $f$ is hyperbolic, the 
subgroup of elements in $G$ that fix each point of  
 $\partial (axis(f))$ is called  the {\it elementary closure} of $f$, denoted by $E(f)$.

\begin{lemma}\label{boundary.axis}
Assume the action of $G$ is acylindrical on $X$. If $g$ fixes one point
in $\partial (axis(f))$, then $g \in E(f)$.
\end{lemma}

\begin{proof}
Let $\gamma$ be a half of $axis(f)$ that tends to the point
fixed by $g$.  Then $|x-gx|$ is bounded for $x \in \gamma$.
Assume that $\gamma$ tends to the direction 
that $f$ translates $axis(f)$ (otherwise, we let $N<0$ below).
Then for any $N>0$ and $x \in \gamma$,
$|f^{-N}gf^N(x)-x|$ is bounded. 
Now by acylindricity (apply it to $x,y \in \gamma$ that are far from each other), there are only finitely many possibilities
for $f^{-N}gf^N$, so $g$ commutes with a nontrivial power of $f$.
So $g$ moves each point in $axis(f)$ by a bounded amount, 
therefore $g \in E(f)$.
\end{proof}

\begin{lemma}\label{projection.bdd}
Assume $f$ is hyperbolic on $X$. If $g \not\in E(f)$,
then $\pi_{axis(f)}M(g)$ is bounded. 
\end{lemma}
\begin{proof}
Suppose not. Let $p$ be a point in  $\partial(axis(f))$ that the projection 
of $M(g)$ tends to. 

We claim $p \in \partial M(g)$.
This is because since 
 both $axis(f)$ and $M(g)$ are quasi-convex,
 a half of $axis(f)$ to the direction of $p$ is contained 
 in a bounded neighborhood of $M(g)$.
 
 So, by Lemma \ref{boundary.min} $g(p)=p$, and by Lemma \ref{boundary.axis}
 $g$ is in $E(f)$, a contradiction. 
 \end{proof}

We will use the following result.
It follows from the assumption that $G$ contains a ``hyperbolically
embedded subgroup'' that is non-degenerate, i.e., proper and infinite, 
see Theorem 1.2 in \cite{Osin}.

By a {\it Schottky subgroup} $F<G$ we mean a free subgroup such that
an orbit map $F\to X$ is a quasi-isometric embedding.

\begin{prop}[{\cite[Theorem 6.14]{DGO}}]\label{schottky}
Suppose the action of $G$ is acylindrical
and $G$ is not virtually cyclic. 

Then $G$ contains a unique maximal finite normal subgroup $K$ and a
Schottky subgroup $F$ so that for every nontrivial $f\in F$ any element $g\in
E(f)$ is either contained in $K$ or has a nontrivial power that
commutes with $f$. If $K$ is trivial, $E(f)$ is cyclic.
\end{prop}

When the order of $g$ is $N<\infty$, we define 
$$MM(g)=\cup_{0<n<N} M(g^n).$$
This set is invariant by $g$, and contains $M(g)$.
The following is obvious from the definition of
the set $M(g^n)$ and $MM(g)$.
\begin{lemma}\label{neighborhood.MM}
For any $x \in X-MM(g)$ and any non-trivial $h \in <g>$, 
we have $|h(x)-x| > 1000 \delta$.
\end{lemma}

\begin{remark}
Although we will not use this fact, we observe that 
if $g$ has finite order $N$, then the set $MM(g)$
is $101 \delta$-quasi-convex.
This is because $g$ has an orbit
whose diameter is at most $6 \delta$ (see the 
proof of Lemma \ref{center}), and 
$MM(g)=\cup_{0<n<N} M(g^n)$ is a union of $100 \delta$-
quasi-convex sets all of which contain the bounded orbit. 
Thus $MM(g)$ is $101 \delta$-quasi-convex as desired. 
(Fix a point $x$ from the bounded orbit. Then for any points $y,z
\in MM(g)$, draw a $\delta$-thin triangle for $x,y,z$. $[y,z]$ is
in the $\delta$-neighborhood of $[x,y]\cup[x,z]$.)
\end{remark}

\begin{lemma}\label{conjugation}
Suppose there is a Schottky free subgroup $F<G$
such that any non-trivial $f \in F$ is hyperbolic and $E(f)$ is cyclic.

Let $S=\{g_1, g_2, g_3, \cdots\}$ be a (finite or infinite) set
of non-trivial elements in $G$.

Then for any given $K>0$ there is a set $S'=\{g_i'\}$ in $G$ such that 
\begin{itemize}
\item
$g_i$ and $g_i'$ are conjugate for each $i$. 

\item
For any $n>0$, 
the sets $M(g_i'), \cdots, M(g_n')$
are $K$-separated,
and moreover, this property holds if we replace $M(g_k')$ with $MM(g_k')$ 
 when $g_k'$ have finite order. 

%

\end{itemize}

\end{lemma}

\begin{proof}

We first prepare a sequence of elements in $F$ that we will 
use to conjugate $g_i$ to $g_i'$.
Take two elements $a,b \in F$ that produce a free subgroup 
of rank two. 
Put $f_i= ab^i, i \ge 1$.

We describe a geometric property we use about the 
sequence of elements. 
Fix a point $x \in X$. Given a constant $L>0$, if we choose
$P$ sufficiently large, then for any $n>0$, and any 
$P_i \ge P$, the following points
are $L$-separated: 
$$x, f_1^{P_1}(x), f_2^{P_2}(x), \cdots, f_n^{P_n}(x).$$
This is an easy consequence of the property such that 
the embedding of the Cayley graph of $F$ in $X$ using 
the orbit of the point $x$ is quasi-isometric to the image. 

Note that the subsets in the above are $L$-separated
if we replace the point $x$
by a bounded set, possibly taking a larger constant for $P$.

To define  $g_1', g_2', g_3', \cdots$, 
choose a sequence 
$$1 \ne n_1 < n_2 < n_3 < \cdots$$ such that 
for each $i$, 
$g_i \not\in E(f_{n_i})$. This is clearly possible.
Then by Lemma \ref{projection.bdd}, the projection of $M(g_i)$ 
to $axis(f_{n_i})$ is bounded. 
Moreover, if the order of $g_i$ is $N<\infty$, then 
$<g_i> \cap E(f_{n_i}) =1$, therefore the projection 
of $MM(g_i)$ to $axis(f_{n_i})$ is bounded
since each $M(g_i^n), 0<n<N$ has a bounded projection. 

Now take a sequence of sufficiently large constants $L_i >0 $, 
depending on the given constant $K$, and 
put $g_i'=f_{n_i}^{L_i} g_i f_{n_i}^{-L_i}$.
Then for each $n>0$,  the sets $M(g_{1}'), \cdots, M(g_{n}')$
are $K$-separated since 
$M(g_i') = f_{n_i}^{L_i}(M(g_i))$.
Also we can arrange so that the sets remain $K$-separated 
if we  replace $M(g_i')$
with $MM(g_i')$ if the order of $g_i'$ are finite, 
maybe for larger constants $L_i$. 
%
%
%
%
%
\end{proof}
%
%
\noindent
{\it Proof of Theorem \ref{notIG}}.
\\
{\it Case 1}. 
Assume $G$ does not contain any non-trivial finite normal 
subgroup. 

By Proposition \ref{schottky}, there is a Schottky subgroup 
$F<G$ such that any non-trivial element $f\in F$ is hyperbolic and $E(f)$
is cyclic.

Let $K>0$ be a constant from Theorem \ref{remark.torsion} for $Q=100\delta$.
Let $\C=\{g_1, g_2, \cdots \}$ be a set of 
all conjugacy classes of $G$ except for the class for $1$.
Apply lemma \ref{conjugation} to the set $\C$ and the constant  $K$
and obtain a new set $\C'=\{g_i'\}$. 
For each $i>0$, $M(g_i')$ is a non-empty, $g_i'$-invariant,
$100\delta$-quasi-convex subset.

Now, for each $n>1$, the assumptions (b) and (c) of Theorem \ref{remark.torsion} are
satisfied by the subgroups $<g_i'>, 1 \le i \le n$ and  the sets $M(g_i'), 1 \le i \le n$.  
If $g_k'$ has finite order, then take $M(g_k') \subset MM(g_k')$ as the
desired neighborhood.
Then by Lemma  \ref{conjugation} they are $K$-separated,
which implies (b) and (c). 

We claim that  (a) holds for $g_k'$ that has finite order.
But for any point $x \in X-MM(g_k')$, and any
non-trivial $h \in <g_k'>$, we have $|h(x)-x| \ge 1000 \delta$
(Lemma \ref{neighborhood.MM}).
This implies (a).


It now follows from Theorem \ref{remark.torsion}  that the subgroup generated by $g_1', \cdots, g_n'$ is the
free product $<g_1'> * \cdots *<g_n'>$ for each $n>0$.

To finish, first suppose that $\C'$ is a finite set, $\{g_1', \cdots, g_n'\}$.
Then $<g_1'> * \cdots *<g_n'>$ must be a proper subgroup of $G$
(therefore $G$ is not IG) 
since otherwise $G$ contains infinitely many conjugacy classes, 
a contradiction (by our assumption, $G$ is not virtually cyclic).

Second, we assume that $\C'$ is an infinite set in the following. 
%
To argue by contradiction, assume that $G$ is generated by $\C'$.

In $<g_1'>*<g_2'>*<g_3'>$, it is easy to 
choose elements $g_1'', g_2'', g_3''$ such that 
each $g_i''$ is conjugate to $g_i'$, and 
the subgroup generated by $g_1'', g_2'', g_3''$ is a proper
subgroup of $<g_1'>*<g_2'>*<g_3'>$.

Define $\C''$ from $\C'$ by replacing $g_1', g_2', g_3'$
by $g_1'', g_2'', g_3''$. $\C''$ contains all non-trivial conjugacy classes
of $G$.
We claim that the subgroup, $G_1$, generated by $\C''$ is a proper subgroup 
in $G$ (so $G$ is not IG). 
To see that, define a quotient homomorphism from $G$ to the group  $H=<g_1'>*<g_2'>*<g_3'>$ by sending all $g_i', i>3$ to $1$. Then 
the image of $G_1$ is a proper subgroup in $H$, so $G_1$ is a proper subgroup
in $G$.
\\
\\
{\it Case 2}. Assume that $G$ contains a non-trivial finite normal subgroup.

Let $K$ be the maximal finite normal subgroup in $G$.
Then $G'=G/K$ does not contain any non-trivial finite normal subgroup.
Moreover $G'$ is acylindrically hyperbolic group. Probably this fact
is well known to specialists, and we postpone giving an argument till the end
(Proposition \ref{quotient}).

By Case 1, $G'$ contains a proper subgroup $H'$ that contains
all conjugacy classes of $G'$. Let $H<G$ be the pull-back 
of $H'$ by the quotient map $G \to G'$.
Then $H$ is a proper subgroup that contains all conjugacy classes
of $G$, therefore $G$ is not IG. 
\qed

\begin{prop}\label{quotient}
Let $G$ be an acylindrically hyperbolic group and $N<G$ a finite 
normal subgroup. Then $G'=G/N$ is an acylindrically hyperbolic group. 

\end{prop}

\begin{proof}
By assumption $G$ acts on a hyperbolic space $X$ such that 
the action is acylindrical and $G$ is not virtually cyclic.
Moreover we may assume that the action is co-compact.
In fact, we may assume that $X$ is a Cayley graph with 
a certain generating set, which is maybe infinite, \cite[Theorem 1.2]{Osin}.

We will produce a new $G$-graph $Y$ from $X$ such that 
the kernel of the action contains $N$ and that $Y$ and $X$ are quasi-isometric.
This is a desired action for $G'$.

For each $N$-orbit of a vertex of $X$, we assign a vertex of $Y$.
Note that $G$ is transitive on the set of $N$-orbits of vertexes
of $X$, so $G$ acts transitively on the vertex set of $Y$.
Now join two vertices of $Y$ if the distance of the 
corresponding $N$-orbits in $X$ is $1$.
$Y$ is a connected $G$-graph and it is easy to check that $Y$ and 
$X$ are quasi-isometric (by the obvious map sending a vertex $x$
of $X$ to the vertex of $Y$ corresponding to the orbit 
of $x$), so that $Y$ is hyperbolic, and that 
the $G$-action on $Y$ is acylindrical. 
By construction, $N$ acts trivially on $Y$.
\end{proof}

\end{document}